\documentclass[a4paper,dvips,10pt,oneside]{article}
\usepackage[makeroom]{cancel}
\usepackage{graphicx}
\usepackage{epstopdf}
\usepackage{subfig}
\usepackage{indentfirst}
\usepackage{color}
\usepackage{float}
\usepackage{leftidx}
\usepackage{bigints}
\usepackage{amssymb,amsmath,dsfont,url,epsfig}
\usepackage[left=1in, right=1in, top=1.1in, bottom=1.1in, includefoot, headheight=13.6pt]{geometry}
\usepackage[T1]{fontenc}
\usepackage{titlesec, blindtext, color}
\definecolor{gray75}{gray}{0.75}

\newcommand{\sln}{\linespread{1}}
\newcommand*{\email}[1]{\href{mailto:#1}{\nolinkurl{#1}} } 
\titleformat{\chapter}[block]{\LARGE\bfseries\sln}{Chapter \thechapter}{11pt}{\newline\huge\bfseries}
\newtheorem{thm}{Theorem}[section]
\newtheorem{rem}{Remark}[section]
\newtheorem{defn}{Definition}[section]

\newenvironment{proof}{\paragraph{Proof:}}{\hfill$\square$}

\newtheorem{proposition}{Proposition}[section]
\newtheorem{corollary}{Corollary}[section]
\begin{document}
\title{ On Minimal Surfaces of Revolutions Immersed in Deformed Hyperbolic Kropina Space}
\author{Ashok Kumar\footnote{E-mail: ashok241001@bhu.ac.in}, Ranadip Gangopadhyay\footnote{E-mail: ranadip.gangopadhyay1@bhu.ac.in}, and Bankteshwar Tiwari\footnote{E-mail: btiwari@bhu.ac.in}\\DST-CIMS, Banaras Hindu University, Varanasi-221005, India\\ Hemangi Madhusudan Shah\footnote{E-mail: hemangimshah@hri.res.in}\\ Harish-Chandra Research Institute, A CI of Homi Bhabha National Institute,\\ Chhatnag Road, Jhunsi, Prayagraj-211019, India.}
\maketitle
\begin{abstract}
In this paper we consider  three dimensional upper half space $\mathbb{H}^3 $ equipped with various Kropina metrics obtained by deformation of hyperbolic metric of $\mathbb{H}^3$ through $1$-forms and obtain a partial differential equation that characterizes minimal surfaces immersed in it. We prove that such minimal surfaces can only be obtained when the hyperbolic metric is deformed along $x^3$ direction. Then we classify such minimal surfaces and show that flag curvature of these surfaces is always non-positive. We also obtain the geodesics of this surface. In particular, it follows that such surfaces neither  have forward conjugate points nor they are forward complete.  
\end{abstract}
\textbf{keywords}: Kropina space, hyperbolic metric, minimal surface, surface of revolution.\hspace{0.5cm}

\par \textbf{ams subject classification 2020}: 53A10, 53B40
\section{Introduction}
The minimal surfaces in Riemannian geometry have been studied by several researchers extensively. Especially, the study of minimal surfaces in  hyperbolic spaces is an active area of research. Particularly, in case of $\mathbb{H}^3$ (upper half space with hyperbolic metric), many new and interesting results  have been obtained in the last few decades. Mori studied complete regular minimal surfaces of revolution in $\mathbb{H}^3$ and their global stability \cite{HM}. do Carmo classified all the complete regular minimal surfaces of revolution in $\mathbb{H}^3$ \cite{DCDM}. The investigation of  minimal surfaces in Finsler geometry, has been started  by Shen \cite{ZS2}. Shen has introduced the notion of mean curvature for immersions into Finsler manifolds and has established some of its properties. The immersion is said to be minimal if the mean curvature of an immersed submanifold of a Finsler manifold is identically zero. 
\par The Randers metric, defined  as $F=\alpha+\beta$, where $\alpha$ is a Riemannian metric and $\beta$ is a non zero $1$-form, is the simplest class of non-Riemannian Finsler metric. Minimal surfaces in a Randers space have been studied by several authors. Souza et. al., have studied the minimal surfaces of revolution and the surfaces defined by the graphs of functions in Minkowski space $\mathbb{R}^3$ with Randers metric; and have obtained some interesting results \cite{MK,MJK}. Minimal surfaces in various Randers spaces have been thoroughly investigated  by Geometers, see e.g., \cite{NC1,NC2,NC3,RK}. A class of  minimal surfaces with Matsumoto metrics and Kropina metrics have been recently investigated in \cite{RGBT} and \cite{RGAKBT} respectively.
\par It is well known that the Randers spaces of constant flag curvature have been classified by Bao et. al., using Zermelo Navigation problem \cite{BRS}. They prove that the solution  of a Zermelo's Navigation problem in a Riemannian manifold $(M,h)$ with a time independent vector field $W$ satisfying $h(W,W)<1$, gives a unique Randers metric on the manifold $M$ and conversely \cite{BRS}. It is an important research problem to know, to what extend Riemannian results can be extended to the Finslerian settings. The study of Randers spaces is a prime topic of research towards such questions.
 \par Another important class of Finsler metrics of $(\alpha,\beta)$ type is Kropina metrics, introduced by Kropina \cite{VK}. These metrics also have numerous applications in physics and biology \cite{PLA}. For a vector field $W$, satisfying $h(W,W)=1$, Yoshikawa and  Sabau have showed that the time minimizing path would be the geodesics of a Kropina metric which is given by $F=\frac{\alpha^2}{\beta}$, where $\alpha$ is a Riemannian metric and $\beta$ is a $1$-form \cite{YR}. They have also proved  that a Kropina space is of constant flag curvature if and only if its Zermelo's navigation data is a unit length Killing vector field on a constant sectional curvature Riemannian manifold. Interestingly, unlike the Randers case, only the odd dimensional spheres and the Euclidean plane admit Kropina metrics of constant flag curvature \cite{SVS}. Cheng et al., \cite{CMM}, have affirmed that the chains in CR geometry are geodesics of a certain Kropina metric constructed from CR structure. In \cite{RGAKBT}, authors study the minimal surfaces immersed in Kropina Minkowski space $\mathbb{R}^3$,  constructed as a slight deformation of Euclidean space by a $1$-form, and completely classify such surfaces under different immersions.

\par Let us denote $\mathbb{H}^3=\left\lbrace (\tilde{x}^1,\tilde{x}^2,\tilde{x}^3)\in \mathbb{R}^3 : \tilde{x}^3>0\right\rbrace $, i.e., upper half space with the hyperbolic metric,
\begin{equation}\label{eqn2.1.1.1}
\tilde{\alpha}=\frac{\sqrt{(d \tilde{x}^1)^2+(d \tilde{x}^2)^2+(d \tilde{x}^3)^2}}{\tilde{x}^3}.
\end{equation}
 In this paper, we study the minimality of surface of revolution as an immersion in  $\mathbb{H}^3$, with Kropina metric $\tilde{F}=\tilde{\alpha}^2/\tilde{\beta}$, where $\tilde{\alpha}$ is the hyperbolic metric and $\tilde{\beta}$ is a $1$-form. Among various deformations of hyperbolic metric by $1$-forms to obtain Kropina metric, we find that only when we consider deformation along $x^3$-axis the immersion is minimal. Indeed we classify all such immersions and study the geometric properties of such surfaces. More precisely, we obtain the following results:
\begin{thm}\label{thm2.1}
Let $(\mathbb{H}^3, \tilde{F}=\tilde{\alpha}^2/\tilde{\beta})$ be a Kropina space, where $\tilde{\alpha}$ is the hyperbolic metric given by  \eqref{eqn2.1.1.1} and $\tilde{\beta} = \frac{bd\tilde{x}^1}{(\tilde{x}^3)^2}$, ($b\ne 0$ is a constant), is a $1$-form. Let $\varphi :M^2 \subset \mathbb{R}^2 \to \mathbb{H}^3$, be an immersion given by
\begin{equation}\label{eqn2.1.1}
\varphi(x^1,x^2) = (f(x^1) \cos x^2, f(x^1) \sin x^2, x^1),
\end{equation}
where $f$ is a smooth positive real valued function defined on $\mathbb{R}_{>0}$. Then $\varphi$ can never be minimal.
 \end{thm}
In the sequel we assume $f$ as in the Theorem \ref{thm2.1}.
\begin{thm}\label{thm2.2}
Let $(\mathbb{H}^3, \tilde{F}=\tilde{\alpha}^2/\tilde{\beta})$ be a Kropina space, where $\tilde{\alpha}$ is the hyperbolic metric given by  \eqref{eqn2.1.1.1} and $\tilde{\beta} = \frac{b(\cos \theta d\tilde{x}^1+\sin \theta d\tilde{x}^2)}{(\tilde{x}^3)^2}$, ($b\ne 0$,~$0\leq\theta<2\pi$ are  constants), is a $1$-form. Let $\varphi$ be given by  $\eqref{eqn2.1.1}$. Then $\varphi$ can not be minimal.
\end{thm}
\begin{thm}\label{thm2.3}
Let $(\mathbb{H}^3, \tilde{F}=\tilde{\alpha}^2/\tilde{\beta})$ be a Kropina space, where  $\tilde{\alpha}$ is the hyperbolic metric given by  \eqref{eqn2.1.1.1} and $\tilde{\beta} = \frac{bd\tilde{x}^3}{(\tilde{x}^3)^2}$, ($b\ne 0$ is a constant), is a $1$-form and  $\varphi$ given by   \eqref{eqn2.1.1}. Then $\varphi$ is minimal if and only if it an is an open cone with the generating line $f(x^1)=\dfrac{1}{\sqrt{2}}x^1+c$,  here $c$ being a real constant. Further, in local coordinates, the flag curvature and S-curvature at a point $x$ of the cone in direction of $y$ are given by
\begin{equation*}\label{eqn2.1.2}
 K(x,y)=\frac { -6b^2 (y^1)^6(y^2)^2}{\left[ 2(y^1)^2+(x^1)^2(y^2)^2\right]^4 } \qquad \textnormal{and} \qquad S(x,y)=\frac { -3x^1y^1(y^2)^2}{ 2(y^1)^2+(x^1)^2(y^2)^2 },
\end{equation*}
respectively. Here $x=(x^1,x^2)\in M$ and $y=(y^1,y^2)\in T_xM \setminus\left\lbrace 0\right\rbrace$.\\
Moreover, the equations of the geodesics starting from a point $x=(x^1,x^2)$ on the chart along the coordinate directions are given by 
\begin{equation}\label{eqn2.1.3}
x^1= c_1t+c_2,~ x^2=k_1 \qquad \textnormal{and} \qquad x^1= k_2,~ x^2=c_3t+c_4,
\end{equation}
where $c_1,c_2,c_3,c_4,k_1,k_2$ are real constants.
\end{thm}
From the above theorem we obtain the following corollaries immediately:
\begin{corollary}\cite{DSZ}
The flag curvature is non-positive for all directions on the surface and hence the surface has no forward conjugate points with respect to the given Kropina metric. 
\end{corollary}
\begin{corollary}
From the equations of geodesics obtained in Theorem \ref{thm2.3}, it is interesting to note that the profile curve, i.e., the generating line of the cone, as well as circles orthogonal to the axis of the cone are geodesics with respect to the Kropina metric. Although it is well known that such circles are not geodesics with respect to the metric on the cone induced from the standard Euclidean metric.
\end{corollary}
Following remarks are in order:
\begin{rem}\textnormal{
\begin{enumerate}
\item The surface obtained in Theorem \ref{thm2.3} is not forward complete, because the vertex is singular point and hence the coordinate  geodesics \eqref{eqn2.1.3} could not be defined for all time. 
\item From the expression of flag curvature in Theorem \ref{thm2.3}, it can be observed that  the flag curvature is symmetric with respect to the direction,  i.e., $ K(x,y)= K(x,-y)$; eventhough the induced metric on the surface is not symmetric, i.e.,   $ F(x,y)\neq F(x,-y)$.  
\end{enumerate}}
\end{rem}
\section{Preliminaries}
 Let $ M $ be an $n$-dimensional smooth manifold. $T_{x}M$ denotes the tangent space of $M$
 at $x$. The tangent bundle of $ M $ is the disjoint union of tangent spaces $ TM:= \sqcup _{x \in M}T_xM $. We denote the elements of $TM$ by $(x,y)$ where $y\in T_{x}M $ and $TM_0:=TM \setminus\left\lbrace 0\right\rbrace $.
 \begin{defn}
 \cite{SSZ} A Finsler metric on $M$ is a function $F:TM \to [0,\infty)$ satisfying the following:
 \\(i) $F$ is smooth on $TM_{0}$. 
 \\(ii) $F$ is a positively 1-homogeneous on the fibers of the tangent bundle $TM$, i.e., $F(\lambda y)=\lambda F(y)$,~ $\lambda >0$ and $y\in T_{x}M $.
 \\(iii) The Hessian of $\frac{F^2}{2}$ in the direction of $y$ i.e., $g_{ij}=\frac{1}{2}\frac{\partial ^2F^2}{\partial y^i \partial y^j}$, is positive definite on $TM_0$.\\
 The pair $(M,F)$ is called a Finsler space and $g=\left( g_{ij}\right) $ is called the fundamental tensor.
 \end{defn}
\par For an $n$-dimensional Finsler manifold $(M^n,F)$, the  Busemann-Hausdorff  volume form is defined as $dV_{BH}=\sigma_{BH}(x)dx$, where
    \begin{equation}\label{eqn2.2.4}
    \sigma_{BH}(x)=\frac{vol(B^n(1))}{vol\left\lbrace (y^i)\in T_xM : F(x,y)< 1 \right\rbrace },
    \end{equation}
    $B^n(1)$ is the Euclidean unit ball in $\mathbb{R}^n$ and $vol$ is the Euclidean volume.
      \begin{proposition}\label{prop1}\cite{XZ}
      Let $F=\alpha\phi(s)$, $s=\beta/\alpha$, be an $(\alpha,\beta)$-metric on an $n$-dimensional manifold $M$. Then the Busemann-Hausdorff volume form $dV_{BH}$ of the $(\alpha,\beta)$-metric $F$ is given by
      \begin{equation}\label{eqn2.2.401}
      dV_{BH}=\frac{\int\limits_{0}^{\pi}\sin^{n-2}(t)dt}{\int\limits_{0}^{\pi}\frac{\sin^{n-2}(t)}{\phi(b\cos (t))^n}dt} dV_{\alpha},
      \end{equation} 
      where, $dV_{\alpha}=\sqrt{det(a_{ij})}dx$ denotes the volume form of the Riemannian metric  $\alpha$.
      \end{proposition}

\par Let $( \tilde{M}^m, \tilde{F})$ be a Finsler manifold, with local coordinates $(\tilde{x}^1, \dots, \tilde{x}^m)$ and 
$\varphi : M^n \to (\tilde{M}^m, \tilde{F})$ be an immersion. Then $\tilde{F}$ induces a Finsler metric $F$
on $M$, defined by
\begin{equation}\label{eqn2.2.5}
F(x,y)=\left( \varphi^*\tilde{F}\right) (x,y)=\tilde{F}\left( \varphi(x),\varphi_*(y)\right) ,\quad \forall (x,y)\in TM.
\end{equation} 
\par The following convention is in order: the greek letters $\epsilon, \eta, \gamma, \tau, \dots$ are the indices ranging from $1$ to $n$ and the latin letters $i,j,k,l,\dots$ are the indices ranging from $1$ to $n+1$.
 \par Let $M^n$ has local coordinates $x = (x^{\epsilon})$, and
$\varphi(x) =\left(  \varphi^i(x^{\epsilon})\right)\in \tilde{M}^{n+1} $. Let
\begin{equation}\label{eqn2.2.7}
D^n_x(1)=\left\lbrace (y^1,y^2,...,y^n)\in T_xM^n:F(x,y)<1\right\rbrace, \quad y= y^{\epsilon}z^i_{\epsilon} \quad \textnormal{and} \quad z=\left( z^i_{\epsilon}\right)=\left( \frac{\partial \varphi^i}{\partial x^{\epsilon}}\right). 
\end{equation}
Following the notations by Shen \cite{ZS2}, if we denote the volume form  on $M^n$ as
\begin{equation}\label{eqn2.2.6}
\mathcal{F}(x,z)=\frac{vol (B^n(1))}{vol (D^n_x(1))}.
\end{equation}
 Then mean curvature $\mathcal{H}_{\varphi}$, for the immersion $\varphi$ along the vector $v$, is given by
 \begin{equation}\label{eqn2.1.8}
 \mathcal{H}_{\varphi}(v)=\frac{1}{\mathcal{F}}\left\lbrace \frac{\partial^2 \mathcal{F}}{\partial z^i_{\epsilon}\partial z^j_{\eta}} \frac{\partial^2 \mathcal{\varphi}^j}{\partial x^{\epsilon}\partial x^{\eta}}+\frac{\partial^2 \mathcal{F}}{\partial z^i_{\epsilon}\partial \tilde{x}^j} \frac{\partial \mathcal{\varphi}^j}{\partial x^{\epsilon}} -\frac{\partial \mathcal{F}}{\partial \tilde{x}^i}\right\rbrace v^i,
 \end{equation}
 where $v=(v^i)$ is a vector field on $\mathbb{H}^{n+1}$.  $\mathcal{H}_{\varphi}(v)$ depends linearly on $v$ and the mean curvature vanishes on $\varphi_*(TM)$. Since, $(\mathbb{H}^{n+1} ,F)$ is a Minkowski space, therefore $F$ is independent from point $x$ and only depend on direction $y$. Hence, the expression of the mean curvature reduces to
 \begin{equation}\label{eqn2.2.9}
 \mathcal{H}_{\varphi}(v)=\frac{1}{\mathcal{F}}\left\lbrace \frac{\partial^2 \mathcal{F}}{\partial z^i_{\epsilon}\partial z^j_{\eta}} \frac{\partial^2 \mathcal{\varphi}^j}{\partial x^{\epsilon}\partial x^{\eta}}\right\rbrace v^i.
 \end{equation}
 The immersion $\varphi$ is said to be minimal when $\mathcal{H}_{\varphi}=0$.
 \par A smooth curve $\gamma$ in a Finsler manifold $M^n$ is a geodesic if and only if $\gamma(t)=(x^i(t))$ satisfies the differential equations:
  		 \begin{equation}\label{eqn2.1.13}
  		 \frac{d^2x^i(t)}{dt^2}+{G}^i\left(\gamma,\frac{d \gamma}{dt}\right)=0, \qquad 1 \le i \le n.
  		 \end{equation}
   Here $G^i=G^i(x,y)$ are local functions on $TM$ called spray coefficients defined by
  	 		 \begin{equation}\label{eqn2.1.14}
  	 		 {G}^i=\frac{1}{4}{g}^{i{\ell}}\left\{\left[{F}^{2}\right]_{x^ky^{\ell}}y^k-	\left[{F}^{2}\right]_{x^{\ell}}\right\}.
  	 		 \end{equation}
 
 \begin{defn}\label{def2}
 	\textnormal{The Riemann curvature	$R={{R}_y:T_xM \rightarrow T_xM}$, for a Finsler space $(M^n,F)$ is defined by
 ${R}_{y}(u)={R}^{i}_{k}(x,y)u^{k} \frac{\partial}{\partial x^i}$,
 	\ $u=u^k\frac{\partial}{\partial x^k}$,\ where ${R}^{i}_{k}={R}^{i}_{k}(x,y)$ denote the coefficients of the Riemann curvature of $F$ and are given by,
 	\begin{equation}\label{eqn2.2.10}
 			{R}^{i}_{k}=2\frac{\partial{G}^i}{\partial	x^k}-y^j\frac{\partial^2{G}^i}{\partial x^j\partial	y^k}+2{G}^j\frac{\partial^2 {G}^i}{\partial y^j\partial	y^k}-\frac{\partial{G}^i}{\partial y^j}\frac{\partial {G}^j}{\partial y^k}.
 	\end{equation}}
 	\end{defn}
 	\par	The flag curvature $K=K(x,y,P)$, generalizes the sectional curvature in Riemannian geometry to the Finsler geometry and does not depend on whether one is using the Berwald, the Chern or the Cartan connection.

 	\begin{defn}\textnormal{	For a tangent plane $P\subset T_xM$ containing a non-zero vector $y$, the flag curvature $K(x,y,P)$ with the pole vector $y$ is defined by
 	\begin{equation}\label{eqn2.1.11}
 	K(x,y,P) :=\frac{g_y(R_y(u), u)}{g_y(y, y)g_y(u, u)- g_y(y, u)g_y(y, u)},
 	\end{equation}
 	where $u\in P$ is such that $P=\text{span} \left\{ y,u\right\} $.\\
 	 If $K(x,y,P)=K(x,y)$, then the	Finsler metric $F$ is said to be of scalar flag curvature and if $K(x,y,P)=constant$, then the Finsler metric $F$ is said to be of constant flag curvature .}
 		\end{defn}
The relation between  flag curvature $K(x,y)$ and the Riemann curvature  $R^i_j$ of a Finsler metric $F$ is given by
 		 \begin{equation}\label{eqn2.1.12}
 		 R^i_j= K(x,y)\left\lbrace F^2 \delta^i_j-FF_{y^j}y^i\right\rbrace.
 		 \end{equation}
\par Further volume form $dV_F=\sigma_F(x) ~dx^1\wedge dx^2 \wedge ...\wedge dx^n$ on Finsler manifold $(M,F)$, the distortion $\tau_F$ is defined by 
\begin{equation*}
\tau _F(x,y) :=\ln \frac{\sqrt{\det(g_{ij}(x,y))}}{\sigma _F(x)}.
\end{equation*}
Now we define S-curvature of the Finsler manifold $(M,F)$ with respect to the volume form $dV_F$.
 		 	\begin{defn}
 				\textnormal{	For a vector $y\in T_xM\backslash \left\lbrace 0\right\rbrace$, let $\gamma=\gamma(t)$ be the geodesic with $\gamma(0)=x$ and $\dot{\gamma}(0)=y$.
 		Then the $S$-curvature of the Finsler metric $F$ is defined by 	
 		\begin{equation*}
 		S(x, y)=\frac{d}{dt}\left[ \tau_{F}\left(\gamma(t), \dot{\gamma}(t)\right) \right]|_{t=0}.
 		\end{equation*}}
 		\end{defn}
 		S-curvature of $F$ in  terms of spray coefficients is given by
 		 		\begin{equation}\label{eqn2.1.15}
 		 		S(x,y)=\frac{\partial G^m}{\partial y^m}-y^m\frac{\partial\left( \ln\sigma_F\right) }{\partial x^m},
 		 		\end{equation}
 		 		where $G^m$ are given by  \eqref{eqn2.1.14}.\\
Now, let $\mathbb{H}^3=\left\lbrace (\tilde{x}^1,\tilde{x}^2,...,\tilde{x}^{n+1})\in \mathbb{R}^{n+1} : \tilde{x}^{n+1}>0\right\rbrace $ with the hyperbolic metric $\tilde{\alpha}=\frac{\sqrt{\delta_{ij} d \tilde{x}^i d \tilde{x}^j}}{\tilde{x}^{n+1}}$. consider a deformation of   the hyperbolic metric $\tilde{\alpha}$ by a $1$-form $\tilde{\beta} =\frac{bl_id\tilde{x}^i}{(\tilde{x}^{n+1})^2}$, ($b\neq 0$ and $l_i \in \mathbb{R}$  are constants, $\sum\limits_{i=1}^{n}(l_i)^2 \neq 0$) to form a Kropina metric $\tilde{F}=\tilde{\alpha}^2/\tilde{\beta}$ on $\mathbb{H}^{n+1}$. It is important to note that Kropina metric is a conic Finsler metric and therefore the Kropina metric $\tilde{F}=\tilde{\alpha}^2/\tilde{\beta}$ is well defined on $\mathbb{H}^{n+1}$. Now we consider the pull-back of  $\tilde{F}$ on an immersed submanifold in $\mathbb{H}^{n+1}$. Thus we have
 		\begin{proposition}\label{ppn1}
 		Let $\varphi :M^n  \to (\mathbb{H}^{n+1}, \tilde{F}=\tilde{\alpha}^2/\tilde{\beta})$, where $\tilde{\alpha}=\frac{\sqrt{\delta_{ij}d\tilde{x}^i d\tilde{x}^j}}{\tilde{x}^{n+1}}$ is the hyperbolic metric on $\mathbb{H}^{n+1}$ and  $\tilde{\beta} =\frac{bl_id\tilde{x}^i}{(\tilde{x}^{n+1})^2}$, ($b\neq 0$ and $l_i \in \mathbb{R}$  are constants, $\sum\limits_{i=1}^{n}(l_i)^2 \neq 0$) be an immersion in a Kropina space with local coordinates $(\varphi^i(x^\epsilon))$. Then the pull back metric $F$ defined by \eqref{eqn2.2.5} is a Kropina metric on $M^n$.
 		\end{proposition}
 		\begin{proof}
 		Let $\varphi:=(\varphi^i(x^\epsilon))$ be an immersion in the Kropina space $(\mathbb{H}^{n
 		+1},\tilde{F})$. Then for any tangent vector $v \in T_xM \setminus\left\lbrace 0\right\rbrace$, using  \eqref{eqn2.2.5} we obtain,  
       \begin{equation*}\label{eqn2.1.150}
 		 		F(v)=\left( \varphi^*(\tilde{F})\right)(v)=\tilde{F}\left(\varphi_*(v) \right) = \frac{\delta_{ij}\frac{\partial \phi^i}{\partial x^{\epsilon}}\frac{\partial \phi^j}{\partial x^{\delta}}dx^{\epsilon}dx^{\delta}}{b l_i\frac{\partial \phi^i}{\partial x^\eta} dx^\eta}=\frac{A_{\epsilon\delta}dx^{\epsilon}dx^{\delta}}{b l_i z^i_\eta dx^\eta},
 		 		\end{equation*}
 		 		where
 		 		\begin{equation}\label{eqn2.3.18}
 		 		 		 		  A=\left( A_{\epsilon \delta}\right)=\left( \sum\limits_{i=1}^{3}z^i_{\epsilon}z^i_{\delta}\right).
 		 		 		 		  \end{equation}
 		 		Hence, $F= \varphi^*(\tilde{F})$ is also a Kropina metric where, $\alpha^2=A_{\epsilon\delta}dx^{\epsilon}dx^{\delta}$, and  $\beta=b l_i z^i_\eta dx^\eta$. 		
 		\end{proof}\\
 Now we restrict our study to immersion in a three dimensional hyperbolic Kropina space $\mathbb{H}^3$.\\ 
 In what follows we will be using the following expressions:\\
 Let $\varphi :M^2  \to (\mathbb{H}^3, \tilde{F}=\tilde{\alpha}^2/\tilde{\beta})$ be as in \eqref{eqn2.1.1}, then
\begin{equation}\label{eqn2.3.28}
\begin{split}
z^i_{\epsilon}:=\frac{\partial \varphi^i}{\partial x^{\epsilon}}=\delta_{\epsilon 1}[\delta_{i1}f'(x^1)\cos x^2+\delta_{i2}f'(x^1)\sin x^2 +\delta_{i3}]\\+\delta_{\epsilon 2}[-\delta_{i1}f(x^1)\sin x^2+\delta_{i2}f(x^1)\cos x^2],
\end{split}
\end{equation}
and
\begin{equation}\label{eqn2.3.280}
\begin{split}
\varphi^j_{x^\epsilon x^\eta}=\left(\delta_{\epsilon 1}\delta_{\eta 2}+\delta_{\epsilon 2}\delta_{\eta 1}\right) f'(x^1)\left( -\delta_{j1}\sin(x^2)+\delta_{j2}\cos(x^2)\right) \\ +\left[ \delta_{\epsilon 1}\delta_{\eta 1}f''(x^1)-\delta_{\epsilon 2} \delta_{\eta 2} f(x^1)\right] \left( \delta_{j1} \cos(x^2)+\delta_{j2} \sin(x^2)\right). 
\end{split}
\end{equation}
In particular,$$ \varphi^3_{x^\epsilon x^\eta} =0;  ~ \forall \epsilon,\eta.$$ 
And
 		\begin{equation}\label{eqn2.3.201}
 		C=\sqrt{\det A}=\sqrt{\sum\limits_{k\neq l}^{}\left( z^k_1\right)^2\left( z^l_2\right)^2-\sum\limits_{k\neq l}^{} z^k_1 z^k_2z^l_1 z^l_2}.
 		\end{equation}
 		Differentiating $C$ with respect to $z^i_{\epsilon}$ we obtain,
 		\begin{equation}\label{eqn2.3.32}
 		\frac{\partial C}{\partial z^i_{\epsilon}}=\frac{1}{C} \sum\limits_{i\neq l}^{}\left( z^i_1 z^l_2- z^i_2 z^l_1\right)\left( \delta_{\epsilon 1}z^l_2-\delta_{\epsilon 2}z^l_1\right). 
 		\end{equation}
 		Hence,
 		\begin{equation}\label{eqn2.3.25}
 		\frac{\partial^2 C^2}{\partial z^i_{\epsilon}\partial z^j_{\eta}}=\frac{\partial}{\partial z^j_{\eta}}\left( 2C\frac{\partial C}{\partial z^i_{\epsilon}}\right)=2\frac{\partial C}{\partial  z^i_{\epsilon}}\frac{\partial C}{\partial z^j_{\eta}}+2C\frac{\partial^2 C}{\partial z^i_{\epsilon} \partial z^j_{\eta}}.
 		\end{equation}
 		Therefore,
 		\begin{equation}\label{eqn2.3.33}
 		\begin{split}
 		\frac{1}{2} \frac{\partial^2 C^2}{\partial z^i_{\epsilon}\partial z^j_{\eta}}= \sum\limits_{i\neq l}^{}\Big[ \delta_{ij}\left( \delta_{\eta 1}z^l_2-\delta_{\eta 2}z^l_1\right)+ \delta_{lj}\left( \delta_{\eta 2}z^i_1-\delta_{\eta 1}z^i_2\right) \Big] \left( \delta_{\epsilon 1}z^l_2-\delta_{\epsilon 2}z^l_1\right) \\  +\sum\limits_{i\neq l}^{}\left( z^i_1z^l_2-z^i_2 z^l_1\right)\delta_{jl} \left( \delta_{\epsilon 1} \delta_{\eta 2}-\delta_{\epsilon 2} \delta_{\eta 1}\right).  
 		\end{split}
 		\end{equation}
 \section{Main Results}
Consider the Kropina metric $\tilde{F}=\tilde{\alpha}^2/\tilde{\beta}$ in upper half space $\mathbb{H}^3$ where $\tilde{\alpha}$ is given by  \eqref{eqn2.1.1.1} and $\tilde{\beta} = \frac{bd\tilde{x}^1}{(\tilde{x}^3)^2}, ~(b\ne 0)$, i.e., a slight deformation of the hyperbolic metric along $\tilde{x}^1$-axis. Now we find the condition for an immersed surface in $\mathbb{H}^3$ to be minimal.
\begin{proposition}\label{P1}
Let $\varphi :M^2 \to (\mathbb{H}^3, \tilde{F}=\tilde{\alpha}^2/\tilde{\beta})$, where $\tilde{\alpha}$ is given by \eqref{eqn2.1.1.1} and  $\tilde{\beta} = \frac{bd\tilde{x}^1}{(\tilde{x}^3)^2}$, ($b\neq 0$ is a constant) be an immersion in a Kropina space with local coordinates $(\varphi^i(x^\epsilon))$. Then $\varphi$ is minimal if and only if 

\begin{equation} \label{eqn2.3.19}
\begin{split}
\frac{\partial^2 \varphi^j}{\partial x^{\epsilon}\partial x^{\eta}}v^i\Bigg[ -2C^2E\frac{\partial^2 E}{\partial z^i_{\epsilon}\partial z^j_{\eta}}+4C^2\frac{\partial E}{\partial z^i_{\epsilon}}\frac{\partial E}{\partial z^j_{\eta}}-6CE\left(\frac{\partial C}{\partial z^i_{\epsilon}}\frac{\partial E}{\partial z^j_{\eta}}+\frac{\partial E}{\partial z^i_{\epsilon}}\frac{\partial C}{\partial z^j_{\eta}} \right)  \\-6CE^2\frac{\partial C}{\partial z^i_{\epsilon}}\frac{\partial C}{\partial z^j_{\eta}}+3E^2\frac{\partial^2 C^2}{\partial z^i_{\epsilon}\partial z^j_{\eta}} & \Bigg]=0,
 \end{split}
\end{equation}
where 
\begin{equation}\label{eqn2.3.20}
E=b^2\sum\limits_{k=1}^{3}(-1)^{\gamma + \tau}z^{k}_{\tilde{\gamma}}z^{k}_{\tilde{\tau}}z^{1}_{\gamma}z^{1}_{\tau}, 
\end{equation}
and ~~ $\tilde{\tau}=\delta_{\tau 2}+2\delta_{\tau 1}$.
\end{proposition}
\begin{proof}
In view of Proposition \ref{prop1}, for a Kropina surface we have $\phi(s)=\frac{1}{s}$ and $n=2$. Therefore, 
\begin{equation}\label{eqn2.3.21}
\begin{split}
 dV_{BH} = \frac{\int\limits_{0}^{\pi}dt}{\int\limits_{0}^{\pi}(b'\cos t)^2dt}\sqrt{\det A} dx=\frac{2}{b'^2}\sqrt{\det A} dx.
\end{split}
\end{equation}
Here, $b'^2=b^2A^{\epsilon\eta}z^1_{\epsilon}z^1_{\eta} =\| \beta \|^2_\alpha$. Hence, the Euclidean volume of $D^2_x(1)$ is 
  \begin{equation}\label{eqn2.3.22}
  vol (D^2_x(1)):=\frac{vol (B^2(1))b^2A^{\epsilon\eta}z^1_{\epsilon}z^1_{\eta}}{2\sqrt{\det A}}.
  \end{equation}
Therefore, from \eqref{eqn2.2.6}, \eqref{eqn2.3.20} and  \eqref{eqn2.3.22}   we have 
\begin{equation}\label{eqn2.3.24}
\mathcal{F}(x,z)=\frac{2C^3}{E}.
\end{equation}
Now differentiating \eqref{eqn2.3.24} twice, first with respect to $z^i_{\epsilon}$, then  with respect to  $z^j_{\eta}$ successively, and using  \eqref{eqn2.3.25} yields  
\begin{equation}\label{eqn2.3.26}
\begin{split}
\frac{\partial^2 \mathcal{F}}{\partial z^i_{\epsilon}\partial z^j_{\eta}}
=\frac{4C^3}{E^3}\frac{\partial E}{\partial z^i_{\epsilon}}\frac{\partial E}{\partial z^j_{\eta}}-\frac{6C^2}{E^2}\left( \frac{\partial C}{\partial  z^j_{\eta} }\frac{\partial E}{ \partial z^i_{\epsilon}}+\frac{\partial E}{\partial  z^j_{\eta} }\frac{\partial C}{ \partial z^i_{\epsilon}}\right) -\frac{2C^3}{E^2}\frac{\partial^2 E}{\partial  z^i_{\epsilon}\partial  z^j_{\eta}}\\+\frac{6C}{E}\frac{\partial C}{\partial z^j_{\eta}}\frac{\partial C}{\partial z^i_{\epsilon}}+\frac{3C}{E}\frac{\partial^2 C^2}{\partial z^i_{\epsilon}\partial z^j_{\eta}}.
\end{split} 
\end{equation}
Using  \eqref{eqn2.3.26} in \eqref{eqn2.2.9} we conclude the proof of the proposition.
\end{proof}
\subsection*{Proof of Theorem \ref{thm2.1}}
\begin{proof}
Consider the surface of revolution given by  immersion $\varphi $ defined by \eqref{eqn2.1.1} with the generating curve $(f(x^1),0,x^1)$. As noted earlier, the mean curvature vanishes on tangent vectors of the immersion $\varphi$. Therefore, we only need to consider a vector field $v$
such that the set $\left\lbrace \varphi_{x^1}, \varphi_{x^2}, v \right\rbrace $ is linearly independent. Thus natural choice of $v$ is along  ${\varphi_{x^1}}\times {\varphi_{x^2}}$, hence we choose
\begin{equation}\label{eqn2.3.290}
v :=(v^1,v^2,v^3)=(-\cos x^2,-\sin x^2, f'(x^1)).
\end{equation}
Now differentiating \eqref{eqn2.3.20} with respect to $z^i_{\epsilon}$ and  with $z^j_{\eta}$ respectively, we get,
\begin{equation}\label{eqn2.3.34}
\frac{\partial E}{\partial z^i_{\epsilon}}=2b^2\Big[z^1_{\tilde{\epsilon}}(-1)^{{\tilde{\epsilon}}+\tau}z^i_{\tilde{\tau}}z^1_{\tau}+\delta_{i1}\sum\limits_{k}^{}(-1)^{\epsilon +\tau}z^1_{\tau}z^k_{\tilde{\tau}}z^k_{\tilde{\epsilon}}\Big],
\end{equation}
\begin{equation}\label{eqn2.3.35}
\begin{split}
\frac{\partial^2 E}{\partial z^i_{\epsilon}\partial z^j_{\eta}}=2b^2\Bigg[\delta_{j1}\delta_{\eta \tilde{\epsilon}}(-1)^{{\tilde{\epsilon}}+\tau}z^i_{\tilde{\tau}}z^1_{\tau}+z^1_{\tilde{\epsilon}}\left\lbrace(-1)^{{\tilde{\epsilon}}+{\tilde{\eta}}} \delta_{ij}z^1_{\tilde{\eta}}+(-1)^{{\tilde{\epsilon}}+\eta} \delta_{j1}z^i_{\tilde{\eta}}\right\rbrace \\ +\delta_{i1}\sum\limits_{k}^{}(-1)^{\epsilon+\eta}\delta_{j1}z^k_{\eta} z^k_{\tilde{\epsilon}}+\delta_{i1}(-1)^{\epsilon+{\tilde{\eta}}}z^1_{\tilde{\eta}}z^j_{\tilde{\epsilon}}+\delta_{i1}(-1)^{\epsilon+\tau} \delta_{\eta \tilde{\epsilon}}z^1_\tau z^j_{\tilde{\tau}}\Bigg].
\end{split}
\end{equation}
In view of \eqref{eqn2.3.18} and \eqref{eqn2.3.20} we  obtain the following:
\begin{equation}\label{eqn2.3.29}
A = 
\begin{pmatrix}
1+(f'(x^1))^2 & 0 \\
0 & (f(x^1))^2 \\
\end{pmatrix}.
\end{equation}
Therefore,
\begin{equation}\label{eqn2.3.30}
C^2=f^2(x^1)(1+f'^2(x^1)),
\end{equation}
\begin{equation}\label{eqn2.3.300}
 E=b^2f^2(x^1)\left[   f'^2(x^1)+ \sin^{2}\left( x^2\right)  \right].
\end{equation}
Contracting  \eqref{eqn2.3.32} and  \eqref{eqn2.3.34} by $v^i$ respectively, and using \eqref{eqn2.3.28}, we obtain
\begin{equation}\label{eqn2.3.36}
\frac{\partial C}{\partial z^i_{\epsilon}}v^i=0,
\end{equation}
\begin{equation}\label{eqn2.3.49}
\frac{\partial E}{\partial z^i_{\epsilon}}v^i= 2b^2f(x^1) \cos\left( x^2\right) \left[ -\delta_{\epsilon 1}f(x^1)f'(x^1) \cos\left( x^2\right) +\delta_{\epsilon 2}\left\lbrace 1 +f'^2 \left( x^1\right) \right\rbrace  \sin\left( x^2\right) \right].
\end{equation}
Using \eqref{eqn2.3.28}, \eqref{eqn2.3.280} together with  \eqref{eqn2.3.32} and  \eqref{eqn2.3.34} we have
\begin{equation}\label{eqn2.3.40}
\frac{\partial C}{\partial z^j_{\eta}}\frac{\partial^2\varphi^j} {\partial x^{\epsilon}\partial x^{\eta}}=\frac{f'(x^1)}{\sqrt{1+f'^2(x^1)}}\delta_{\epsilon 1}\left[ f(x^1)f''(x^1)+1+f'^2(x^1)\right],
\end{equation}
\begin{equation}\label{eqn2.3.41}
\begin{split}
\frac{\partial E}{\partial z^j_{\eta}}\frac{\partial^2\varphi^j} {\partial x^{\epsilon}\partial x^{\eta}}= 2b^2f(x^1)\Big[\delta_{\epsilon 1} f' \left( x^1\right) \left\lbrace f(x^1)f''\left( x^1\right)+f'^2(x^1)+\sin^2\left( x^2\right)\right\rbrace  +\delta_{\epsilon 2} f(x^1) \sin\left( x^2\right) \cos\left( x^2\right) \Big].
\end{split}
\end{equation}
Further, using  \eqref{eqn2.3.28}, \eqref{eqn2.3.280}, \eqref{eqn2.3.33}, \eqref{eqn2.3.290} and \eqref{eqn2.3.35} yields
\begin{equation}\label{eqn2.3.43}
\frac{1}{2}\frac{\partial^2 C^2}{\partial z^i_{\epsilon}\partial z^j_{\eta}}\frac{\partial^2\varphi^j} {\partial x^{\epsilon}\partial x^{\eta}}v^i= f(x^1)\left[ -f(x^1)f''(x^1)+1+f'^2(x^1)\right],
\end{equation}
\begin{eqnarray}\label{eqn2.3.42}
\frac{\partial^2 E}{\partial z^i_{\epsilon}\partial z^j_{\eta}}\frac{\partial^2\varphi^j} {\partial x^{\epsilon}\partial x^{\eta}}v^i=  -2b^2f(x^1)\left[ f(x^1)f''\left( x^1\right) -\cos^2\left( x^2\right)\right].
\end{eqnarray}

Using  \eqref{eqn2.3.30}-\eqref{eqn2.3.42} in  \eqref{eqn2.3.19}, and after some simplifications we get
\begin{equation}\label{eqn2.3.44}
\begin{split}
-3 \left( 1+2f'^2(x^1)\right) \left[  f\left( x^1\right)f''\left( x^1\right)+f'^2\left( x^1\right)+1\right] \cos^4\left( x^2\right)\\+2\left( 1+f'^2(x^1)\right) \left[  2f(x^1)f''(x^1)-f'^2(x^1)\left\lbrace f(x^1)f''(x^1)+f'^2(x^1)+1\right\rbrace\right] \cos^2\left( x^2\right)\\ +\left( 1+f'^2(x^1)\right)\left[ -f(x^1)f''(x^1)+3f'^2(x^1)+3\right]=0.
\end{split}
\end{equation}
The equation \eqref{eqn2.3.44} is an identity in $\cos^2\left( x^2\right) $ and therefore comparing the coefficients of different powers of $\cos^2\left( x^2\right) $  we obtain 
\begin{equation}\label{eqn2.3.45}
         f(x^1)f''(x^1)+f'^2(x^1)+1=0,
\end{equation}
and \begin{equation}\label{eqn2.3.46}
    -f(x^1)f''(x^1)+3f'^2(x^1)+3=0.
\end{equation}
Adding  \eqref{eqn2.3.45} and  \eqref{eqn2.3.46} we get
\begin{equation*}\label{eqn2.3.47}
4\left( f'^2(x^1)+1\right) =0.
\end{equation*}
Since $f'$ is a real valued function we get a contradiction.
\end{proof}
\par Following the strategy developed in the proof of Proposition \ref{P1} and Theorem \ref{thm2.1} above we conclude easily
\begin{corollary}\label{thm2}
Let $(\mathbb{H}^3, \tilde{F}=\tilde{\alpha}^2/\tilde{\beta})$ be a Kropina space, where  $\tilde{\alpha}$ is given by  \eqref{eqn2.1.1.1}, $\tilde{\beta} = \frac{bd\tilde{x}^2}{(\tilde{x}^3)^2}$,  $b\ne 0$, is a $1$-form and $\varphi$ given by \eqref{eqn2.1.1}. Then $\varphi$ can never be minimal.
\end{corollary}
Now we consider the Kropina space $(\mathbb{H}^3, \tilde{F}=\tilde{\alpha}^2/\tilde{\beta})$ where $\tilde{\alpha}$ is given by  \eqref{eqn2.1.1.1}, $\tilde{\beta} = \frac{b(\cos \theta d\tilde{x}^1+\sin \theta d\tilde{x}^2)}{(\tilde{x}^3)^2}$, ($b\ne 0$,~$0\leq\theta<2\pi$ are  constants), and obtain the following:
\begin{proposition}\label{ppn2}
Let $\varphi :M^2  \to (\mathbb{H}^3, \tilde{F}=\tilde{\alpha}^2/\tilde{\beta})$,  where $\tilde{\alpha}$ is given by  \eqref{eqn2.1.1.1} and $\tilde{\beta} = \frac{b(\cos \theta d\tilde{x}^1+\sin \theta d\tilde{x}^2)}{(\tilde{x}^3)^2}$,  ($b\ne 0$,~$0\leq\theta<2\pi$ are  constants)  be an immersion in the Kropina space with local coordinates $(\varphi^i(x^\epsilon))$. Then $\varphi$ is minimal if and only if \eqref{eqn2.3.19} holds, with 
\begin{equation}\label{eqn2.4.49}
\begin{split}
E=b^2\sum\limits_{k=1}^{3}(-1)^{\gamma + \tau}z^{k}_{\tilde{\gamma}}z^{k}_{\tilde{\tau}}\left[\cos^2\theta z^{1}_{\gamma}z^{1}_{\tau}+\sin^2\theta z^{2}_{\gamma}z^{2}_{\tau}+2\sin\theta \cos\theta z^{1}_{\gamma}z^{2}_{\tau}\right].
\end{split}
\end{equation}
\end{proposition}

\begin{proof}
In view of  \eqref{eqn2.3.21} we have
\begin{equation*}\label{eqn2.4.51}
\begin{split}
 dV_{BH} = \frac{2}{b'^2}\sqrt{\det A}dx,
\end{split}
\end{equation*}
where, $b'^2=b^2A^{\gamma\tau }\left[\cos^2\theta z^{1}_{\gamma}z^{1}_{\tau}+\sin^2\theta z^{2}_{\gamma}z^{2}_{\tau}+2z^{1}_{\gamma}z^{2}_{\tau}\sin\theta \cos\theta \right]=\| \beta \|^2_\alpha$, and the Euclidean volume of $D^2_x(1)$ is given by
  \begin{equation}\label{eqn2.4.52}
  vol (D^2_x(1)):=\frac{vol (B^2(1))b^2A^{\gamma\tau}\left[\cos^2\theta z^{1}_{\gamma}z^{1}_{\tau}+\sin^2\theta z^{2}_{\gamma}z^{2}_{\tau}+2\sin\theta \cos\theta z^{1}_{\gamma}z^{2}_{\tau}\right]}{2\sqrt{\det A}}.
  \end{equation}
 
Therefore, from \eqref{eqn2.2.6}, \eqref{eqn2.3.201}, \eqref{eqn2.4.49} and  \eqref{eqn2.4.52}   we have 
\begin{equation}\label{eqn2.4.54}
\mathcal{F}(x,z)=\frac{2C^3}{E}.
\end{equation}
Now by the similar calculations as in the Proposition \ref{P1} we complete the proof of the Proposition.
\end{proof}
\subsection*{Proof of Theorem \ref{thm2.2}}
\begin{proof}
As in Theorem \ref{thm2.1}, we consider $v :=(v^1,v^2,v^3)=(-\cos x^2,-\sin x^2, f'(x^1))$.\\ 
Now differentiating \eqref{eqn2.4.49} with respect to $z^i_{\epsilon}$ and  $z^j_{\eta}$ successively we have,
\begin{equation}\label{eqn2.4.60}
\begin{split}
\frac{\partial E}{\partial z^i_{\epsilon}}=2b^2 \cos^2 \theta \Bigg[z^1_{\tilde{\epsilon}}(-1)^{{\tilde{\epsilon}}+\tau}z^i_{\tilde{\tau}}z^1_{\tau}+\delta_{i1}\sum\limits_{k}^{}(-1)^{\epsilon +\tau}z^1_{\tau}z^k_{\tilde{\tau}}z^k_{\tilde{\epsilon}}\Bigg] \\+2b^2\sin^2 \theta \Bigg[z^2_{\tilde{\epsilon}}(-1)^{{\tilde{\epsilon}}+\tau}z^i_{\tilde{\tau}}z^2_{\tau}+\delta_{i2}\sum\limits_{k}^{}(-1)^{\epsilon +\tau}z^2_{\tau}z^k_{\tilde{\tau}}z^k_{\tilde{\epsilon}}\Bigg] \\+ 2b^2 \cos\theta \sin\theta \Bigg[z^1_{\tilde{\epsilon}}(-1)^{{\tilde{\epsilon}}+\tau} z^i_{\tilde{\tau}} z^2_{\tau} +z^2_{\tilde{\epsilon}}(-1)^{{\tilde{\epsilon}}+\tau} z^i_{\tilde{\tau}} z^1_{\tau}\\ + \delta_{i1}\sum\limits_{k}^{}(-1)^{\epsilon +\tau}z^2_{\tau}z^k_{\tilde{\tau}}z^k_{\tilde{\epsilon}} +\delta_{i2}\sum\limits_{k}^{}(-1)^{\epsilon +\tau}z^1_{\tau}z^k_{\tilde{\tau}}z^k_{\tilde{\epsilon}} \Bigg], 
\end{split}
\end{equation}
and
\begin{eqnarray}\label{eqn2.4.61}
\begin{split}
\frac{\partial^2 E}{\partial z^i_{\epsilon}\partial z^j_{\eta}}=2b^2 \cos^2\theta\Bigg[\delta_{j1}\delta_{\eta \tilde{\epsilon}}(-1)^{{\tilde{\epsilon}}+\tau}z^i_{\tilde{\tau}}z^1_{\tau} + z^1_{\tilde{\epsilon}} \left\lbrace (-1)^{\tilde{\epsilon}+\tilde{\eta}} \delta_{ij} z^1_{\tilde{\eta}} +(-1)^{\tilde{\epsilon}+\eta}\delta_{j1} z^i_{\tilde{\eta}} \right\rbrace \\ +\delta_{i1}\left\lbrace \sum\limits_{k}^{}(-1)^{\epsilon +\eta} \delta_{j1} z^k_{\tilde{\eta}} z^k_{\tilde{\epsilon}} +(-1)^{\epsilon +\tilde{\eta}}   z^1_{\tilde{\eta}} z^j_{\tilde{\epsilon}} + (-1)^{\epsilon +\tau} \delta_{\eta \tilde{\epsilon}} z^j_{\tilde{\tau}}z^1_{\tau} \right\rbrace \Bigg] \\ +2b^2 \sin^2\theta\Bigg[\delta_{j2}\delta_{\eta \tilde{\epsilon}}(-1)^{{\tilde{\epsilon}}+\tau}z^i_{\tilde{\tau}}z^2_{\tau} + z^2_{\tilde{\epsilon}} \left\lbrace (-1)^{\tilde{\epsilon}+\tilde{\eta}} \delta_{ij} z^2_{\tilde{\eta}} +(-1)^{\tilde{\epsilon}+\eta}\delta_{j2} z^i_{\tilde{\eta}} \right\rbrace \\ +\delta_{i2}\left\lbrace \sum\limits_{k}^{}(-1)^{\epsilon +\eta} \delta_{j2} z^k_{\tilde{\eta}} z^k_{\tilde{\epsilon}} +(-1)^{\epsilon +\tilde{\eta}}   z^2_{\tilde{\eta}} z^j_{\tilde{\epsilon}} + (-1)^{\epsilon +\tau} \delta_{\eta \tilde{\epsilon}} z^j_{\tilde{\tau}}z^2_{\tau} \right\rbrace \Bigg] \\ +2b^2 \cos\theta\sin\theta \Bigg[\delta_{j1}\delta_{\eta \tilde{\epsilon}}(-1)^{{\tilde{\epsilon}}+\tau}z^i_{\tilde{\tau}}z^2_{\tau} +z^1_{\tilde{\epsilon}} \left\lbrace (-1)^{\tilde{\epsilon}+\tilde{\eta}} \delta_{ij} z^2_{\tilde{\eta}} +(-1)^{\tilde{\epsilon}+\eta}\delta_{j2} z^i_{\tilde{\eta}} \right\rbrace \\ +\delta_{j2}\delta_{\eta \tilde{\epsilon}}(-1)^{{\tilde{\epsilon}}+\tau}z^i_{\tilde{\tau}}z^1_{\tau} +z^2_{\tilde{\epsilon}} \left\lbrace (-1)^{\tilde{\epsilon}+\tilde{\eta}} \delta_{ij} z^1_{\tilde{\eta}} +(-1)^{\tilde{\epsilon}+\eta}\delta_{j1} z^i_{\tilde{\eta}} \right\rbrace \\ +\delta_{i1} \left\lbrace \sum\limits_{k}^{}(-1)^{\epsilon +\eta} \delta_{j2} z^k_{\tilde{\eta}} z^k_{\tilde{\epsilon}} +(-1)^{\epsilon +\tilde{\eta}}   z^2_{\tilde{\eta}} z^j_{\tilde{\epsilon}} + (-1)^{\epsilon +\tau} \delta_{\eta \tilde{\epsilon}} z^j_{\tilde{\tau}}z^2_{\tau} \right\rbrace \\ + \delta_{i2} \left\lbrace \sum\limits_{k}^{}(-1)^{\epsilon +\eta} \delta_{j1} z^k_{\tilde{\eta}} z^k_{\tilde{\epsilon}} +(-1)^{\epsilon +\tilde{\eta}}   z^1_{\tilde{\eta}} z^j_{\tilde{\epsilon}} + (-1)^{\epsilon +\tau} \delta_{\eta \tilde{\epsilon}} z^j_{\tilde{\tau}}z^1_{\tau} \right\rbrace \Bigg].
\end{split}
\end{eqnarray}
For the surface of revolution under consideration given by  \eqref{eqn2.1.1},  \eqref{eqn2.4.49} is rewritten as
\begin{equation}\label{eqn2.4.590}
E=b^2f^2(x^1)\left[   f'^2(x^1)+ \sin^{2}\left( x^2-\theta\right)  \right].
\end{equation}
Contracting \eqref{eqn2.3.32} and  \eqref{eqn2.4.60} by $v^i$  and using \eqref{eqn2.3.28} we obtain,
\begin{equation}\label{eqn2.4.62}
\frac{\partial C}{\partial z^i_{\epsilon}}v^i=0,
\end{equation}
\begin{equation}\label{eqn2.4.63}
\begin{split}
\frac{\partial E}{\partial z^i_{\epsilon}}v^i= 2b^2f(x^1) \cos\left( x^2-\theta\right) \Big[ -\delta_{\epsilon 1}f(x^1)f'(x^1) \cos\left( x^2-\theta\right) +\delta_{\epsilon 2}\left\lbrace  1+f'^2\left( x^1\right)  \right\rbrace  \sin\left( x^2-\theta\right) \Big].
\end{split} 
\end{equation}
Using  \eqref{eqn2.3.28},  \eqref{eqn2.3.280} together with  \eqref{eqn2.3.32} and  \eqref{eqn2.4.60}, we have
\begin{equation}\label{eqn2.4.64}
\frac{\partial C}{\partial z^j_{\eta}}\frac{\partial^2\varphi^j} {\partial x^{\epsilon}\partial x^{\eta}}=\frac{f'(x^1)}{\sqrt{1+f'^2(x^1)}}\delta_{\epsilon 1}\left[ f(x^1)f''(x^1)+1+f'^2(x^1)\right],  
\end{equation}
\begin{equation}\label{eqn2.4.65}
\begin{split}
\frac{\partial E}{\partial z^j_{\eta}}\frac{\partial^2\varphi^j} {\partial x^{\epsilon}\partial x^{\eta}}= 2b^2f(x^1)\Big[\delta_{\epsilon 1} f'\left( x^1\right) \left\lbrace f(x^1)f''\left( x^1\right)+f'^2(x^1)+\sin\left( x^2-\theta\right)\right\rbrace \\+\delta_{\epsilon 2}f(x^1)\sin\left( x^2-\theta\right)\cos \left( x^2-\theta\right)\Big].
\end{split} 
\end{equation}
Note, from  \eqref{eqn2.3.28}, \eqref{eqn2.3.280}, \eqref{eqn2.3.33}, \eqref{eqn2.3.290} and  \eqref{eqn2.4.61} we obtain
\begin{equation}\label{eqn2.4.67}
\frac{1}{2}\frac{\partial^2 C^2}{\partial z^i_{\epsilon}\partial z^j_{\eta}}\frac{\partial^2\varphi^j} {\partial x^{\epsilon}\partial x^{\eta}}v^i= f(x^1)\left[ -f(x^1)f''(x^1)+1+f'^2(x^1)\right], 
\end{equation}
\begin{eqnarray}\label{eqn2.4.66}
\frac{\partial^2 E}{\partial z^i_{\epsilon}\partial z^j_{\eta}}\frac{\partial^2\varphi^j} {\partial x^{\epsilon}\partial x^{\eta}}v^i=  -2b^2f(x^1)\left[ f(x^1)f''\left( x^1\right) -\cos^2\left( x^2-\theta\right)\right].
\end{eqnarray}
Further, using  \eqref{eqn2.3.30} and   \eqref{eqn2.4.590}-\eqref{eqn2.4.66}, in  \eqref{eqn2.3.19}, we obtain 
\begin{equation}\label{eqn2.4.68}
\begin{split}
-3 \left( 1+2f'^2(x^1)\right) \left[  f\left( x^1\right)f''\left( x^1\right)+f'^2\left( x^1\right)+1\right] \cos^4 \left( x^2-\theta\right)\\+2\left( 1+f'^2(x^1)\right) \left[  2f(x^1)f''(x^1)-f'^2(x^1)\left\lbrace f(x^1)f''(x^1)+f'^2(x^1)+1\right\rbrace\right] \cos^2 \left( x^2-\theta\right)\\+ \left( 1+f'^2(x^1)\right)^2\left[ -f(x^1)f''(x^1)+3f'^2(x^1)+3\right]=0.
\end{split}
\end{equation}
 The equation \eqref{eqn2.4.68} is an identity in $\cos^2\left( x^2-\theta\right) $ and therefore comparing the coefficients of different powers of $\cos^2\left( x^2-\theta\right) $  we obtain,
\begin{equation}\label{eqn2.4.69}
         f(x^1)f''(x^1)+3f'^2(x^1)+3=0,
\end{equation}
and \begin{equation}\label{eqn2.4.70}
    -f(x^1)f''(x^1)+f'^2(x^1)+1=0.
\end{equation}
Adding  \eqref{eqn2.4.69} and  \eqref{eqn2.4.70}, we have
\begin{equation*}\label{eqn2.4.71}
4\left( f'^2(x^1)+1\right) =0.
\end{equation*}
Since $f'$ is a real valued function, we get a contradiction. Hence the proof follows.
\end{proof}

Finally, we consider the Kropina space $(\mathbb{H}^3,\tilde{F}=\tilde{\alpha}^2/\tilde{\beta})$, where $\tilde{\alpha}$ is the hyperbolic metric given by  \eqref{eqn2.1.1.1} and $\tilde{\beta} = \frac{bd\tilde{x}^3}{(\tilde{x}^3)^2}$, with  $(b\ne 0)$ and the minimal surface of revolution immersed in it.
\begin{proposition}
Let $\varphi :M^2 \to (\mathbb{H}^3, \tilde{F}=\tilde{\alpha}^2/\tilde{\beta})$, where $\tilde{\alpha}$ is given by  \eqref{eqn2.1.1.1} and  $\tilde{\beta} = \frac{bd\tilde{x}^3}{(\tilde{x}^3)^2}$, $b\ne 0$, be an immersion in the Kropina space with local coordinates $(\varphi^i(x^\epsilon))$. Then $\varphi$ is minimal if and only if  \eqref{eqn2.3.19} holds with 
\begin{equation}\label{eqn2.5.100}
E=b^2\sum\limits_{k=1}^{3}(-1)^{\gamma + \tau}z^{k}_{\tilde{\gamma}}z^{k}_{\tilde{\tau}}z^{3}_{\gamma}z^{3}_{\tau}. 
\end{equation}
\end{proposition}
\begin{proof}
In view of  \eqref{eqn2.3.21} the Busemann Hausdorff volume form is given by
\begin{equation*}\label{eqn2.5.101}
\begin{split}
 dV_{BH}=\frac{2}{b'^2}\sqrt{\det A}dx,
\end{split}
\end{equation*}
where, $b'^2=b^2A^{\epsilon\eta}z^3_{\epsilon}z^3_{\eta}=\|\beta\|^2_\alpha$. Hence, the Euclidean volume of $D^2_x(1)$ is given by
  \begin{equation}\label{eqn2.5.102}
  vol (D^2_x(1)):=\frac{vol (B^2(1))b^2A^{\epsilon\eta}z^3_{\epsilon}z^3_{\eta}}{2\sqrt{\det A}}.
  \end{equation}
  
Therefore, from  \eqref{eqn2.2.6}, \eqref{eqn2.3.201}, \eqref{eqn2.5.100} and  \eqref{eqn2.5.102}  we have,
\begin{equation}\label{eqn2.5.104}
\mathcal{F}(x,z)=\frac{2C^3}{E}.
\end{equation}
A similar computations as in  Proposition \ref{P1} completes the proof.
\end{proof}
\subsection*{Proof of Theorem \ref{thm2.3}}
\begin{proof}
\par As in the earlier case, we consider $v :=(v^1,v^2,v^3)=(-\cos x^2,-\sin x^2, f'(x^1))$.
Now differentiating  \eqref{eqn2.5.100} with respect to $z^i_{\epsilon}$ and with $z^j_{\eta}$ respectively, we get 
\begin{equation}\label{eqn2.5.108}
\frac{\partial E}{\partial z^i_{\epsilon}}=2b^2\Big[z^3_{\tilde{\epsilon}}\sum\limits_{\tau}^{}(-1)^{{\tilde{\epsilon}}+\tau}z^i_{\tilde{\tau}}z^3_{\tau}+\delta_{i3}\sum\limits_{\tau ,k}^{}(-1)^{\epsilon +\tau}z^3_{\tau}z^k_{\tilde{\tau}}z^k_{\tilde{\epsilon}}\Big],
\end{equation}
\begin{equation}\label{eqn2.5.109}
\begin{split}
\frac{\partial^2 E}{\partial z^i_{\epsilon}\partial z^j_{\eta}}=2b^2\Bigg[\delta_{j3}\delta_{\eta \tilde{\epsilon}}(-1)^{{\tilde{\epsilon}}+\tau}z^i_{\tilde{\tau}}z^3_{\tau}+z^3_{\tilde{\epsilon}}\left\lbrace(-1)^{{\tilde{\epsilon}}+{\tilde{\eta}}} \delta_{ij}z^3_{\tilde{\eta}}+(-1)^{{\tilde{\epsilon}}+\eta} \delta_{j3}z^i_{\tilde{\eta}}\right\rbrace \\ +\delta_{i3}\sum\limits_{k}^{}(-1)^{\epsilon+\eta}\delta_{j3}z^k_{\eta} z^k_{\tilde{\epsilon}}+\delta_{i3}(-1)^{\epsilon+{\tilde{\eta}}}z^3_{\tilde{\eta}}z^j_{\tilde{\epsilon}}+\delta_{i3}(-1)^{\epsilon+\tau} \delta_{\eta \tilde{\epsilon}}z^3_\tau z^j_{\tilde{\tau}}\Bigg].
\end{split}
\end{equation}
For the surface of revolution given in   \eqref{eqn2.1.1},   \eqref{eqn2.5.100} can be rewritten as 
\begin{equation}\label{eqn2.5.1070}
E=b^2f^2(x^1).
\end{equation}
Contracting  \eqref{eqn2.3.32} and \eqref{eqn2.5.108} by $v^i$ respectively, and using  \eqref{eqn2.3.28} we obtain
\begin{equation}\label{eqn2.5.110}
\frac{\partial C}{\partial z^i_{\epsilon}}v^i=0,
\end{equation}
\begin{equation}\label{eqn2.5.112}
\frac{\partial E}{\partial z^i_{\epsilon}}v^i= 2b^2\delta_{\epsilon 1}f^2(x^1)f'(x^1).
\end{equation}
Using  \eqref{eqn2.3.28},  \eqref{eqn2.3.280} together with  \eqref{eqn2.3.32}, \eqref{eqn2.5.108} we have
\begin{equation}\label{eqn2.5.113}
\frac{\partial C}{\partial z^j_{\eta}}\frac{\partial^2\varphi^j} {\partial x^{\epsilon}\partial x^{\eta}}=\frac{f'(x^1)}{\sqrt{1+f'^2(x^1)}}\delta_{\epsilon 1}\left[ f(x^1)f''(x^1)+1+f'^2(x^1)\right],  
\end{equation}
\begin{equation}\label{eqn2.5.114}
\frac{\partial E}{\partial z^j_{\eta}}\frac{\partial^2\varphi^j} {\partial x^{\epsilon}\partial x^{\eta}}= 2b^2\delta_{\epsilon 1}f(x^1)f'(x^1).
\end{equation}
Further, from  \eqref{eqn2.3.28}, \eqref{eqn2.3.280}, \eqref{eqn2.3.33}, \eqref{eqn2.3.290} and \eqref{eqn2.5.109}, we have
\begin{equation}\label{eqn2.5.116}
\frac{1}{2}\frac{\partial^2 C^2}{\partial z^i_{\epsilon}\partial z^j_{\eta}}\frac{\partial^2\varphi^j} {\partial x^{\epsilon}\partial x^{\eta}}v^i= f(x^1)\left[ -f(x^1)f''(x^1)+1+f'^2(x^1)\right],
\end{equation}
\begin{equation}\label{eqn2.5.115}
\frac{\partial^2 E}{\partial z^i_{\epsilon}\partial z^j_{\eta}}\frac{\partial^2\varphi^j} {\partial x^{\epsilon}\partial x^{\eta}}v^i=  2b^2f(x^1)\left[ 1+2f'^2(x^1)\right].
\end{equation}
Using  \eqref{eqn2.3.30}  and  \eqref{eqn2.5.110}-\eqref{eqn2.5.115} in  \eqref{eqn2.3.19}, and after some simplifications we obtain,
\begin{equation}\label{eqn2.5.1111}
\left(1+2f'^2(x^1)\right)\left\lbrace f'^2(x^1)+3f(x^1)f''(x^1) \right\rbrace -1 =0.
\end{equation}
The above differential equation has only one real solution and  is given by 
\begin{equation}\label{eqn2.5.111}
f(x^1)=\frac{x^1}{\sqrt{2}}+c,
\end{equation}
where $ c \in \mathbb{R}$ is a constant.
\par Without loss of generality we can assume $c=0$ in  \eqref{eqn2.5.111}. In this case we obtain the pullback Finsler metric as
\begin{equation}\label{eqn2.5.117}
F(x,y)=\frac{2(y^1)^2+(x^1)^2(y^2)^2} {by^1},
\end{equation}
 where $x=(x^1,x^2)\in M$ and $y=(y^1,y^2)\in TM \setminus\left\lbrace 0\right\rbrace$. The fundamental metric tensor of the Finsler metric $F$ is given by
 \begin{equation}\label{eqn2.5.118}
 g= (g_{ij})=
 \begin{pmatrix}
 \frac{\left\lbrace a^2 (y^1)^4+3\tilde{b}^2(y^2)^4\right\rbrace }{(y^1)^4} & \frac{-4\tilde{b}^2(y^2)^3}{(y^1)^3} \\
 \frac{-4\tilde{b}^2(y^2)^3}{(y^1)^3} &\frac {2\tilde{b}\left\lbrace a (y^1)^2+3\tilde{b}^2(y^2)^2\right\rbrace }{(y^1)^2}
 \end{pmatrix},
 \end{equation}
 where $a=\frac{2}{b}$ and $\tilde{b}=\frac{(x^1)^2}{b}$. Further, the inverse metric tensor $g^{-1}=(g^{ij})$ of $F$ is given by
 \begin{equation}\label{eqn2.5.119}
g^{-1}= (g^{ij})= \frac{1}{\left[ a(y^1)^2+\tilde{b}^(y^2)^2\right]^3 } 
 \begin{pmatrix}
 (y^1)^4\left\lbrace a (y^1)^2+3\tilde{b}(y^2)^2\right\rbrace  & 2\tilde{b}(y^1)^3 (y^2)^3 \\
 2\tilde{b}(y^1)^3 (y^2)^3 &\frac{(y^1)^2 }{2\tilde{b}}\left\lbrace a^2(y^1)^4+3\tilde{b}^2(y^2)^4\right\rbrace \\
 \end{pmatrix}.
 \end{equation}
 Using  \eqref{eqn2.1.12} the coefficients of the Riemannian curvature tensor are obtained as
 \begin{equation}\label{eqn2.5.121}
 \begin{split}
 R^1_1=\frac { -12(x^1)^2 (y^1)^4(y^2)^4}{\left[ 2(y^1)^2+(x^1)^2(y^2)^2\right]^3 }, \quad R^1_2=\frac { 12(x^1)^2 (y^1)^5(y^2)^3}{\left[ 2(y^1)^2+(x^1)^2(y^2)^2\right]^3 },\hspace{3.0cm} \\ R^2_1=\frac {-6(\left[ -2(y^1)^2+(x^1)^2(y^2)^2\right] (y^1)^3(y^2)^3}{\left[ 2(y^1)^2+(x^1)^2(y^2)^2\right]^3 }, \quad R^2_2=\frac { 6(\left[ -2(y^1)^2+(x^1)^2(y^2)^2\right] (y^1)^4(y^2)^2}{\left[ 2(y^1)^2+(x^1)^2(y^2)^2\right]^3 }. 
 \end{split}
 \end{equation}
 From  \eqref{eqn2.1.12}  the flag curvature of the surface with flag pole $y(\neq 0)$ is given by
 \begin{equation}\label{eqn2.5.122}
  K(x,y)=\frac { -6b^2 (y^1)^6(y^2)^2}{\left[ 2(y^1)^2+(x^1)^2(y^2)^2\right]^4 }.
 \end{equation}
 Therefore from  \eqref{eqn2.1.15} and   \eqref{eqn2.5.120}  we obtain,
 \begin{equation*}
 S(x,y)=\frac { -3x^1y^1(y^2)^2}{ 2(y^1)^2+(x^1)^2(y^2)^2 }.
 \end{equation*}
Using  \eqref{eqn2.1.14} we also obtain the spray coefficients of $F$ as
\begin{equation}\label{eqn2.5.120}
G^1=\frac { -x^1(y^1)^2(y^2)^2}{\left[ 2(y^1)^2+(x^1)^2(y^2)^2\right] }, \quad G^2=\frac {2(y^1)^3(y^2)}{x^1\left[ 2(y^1)^2+(x^1)^2(y^2)^2\right] }.
\end{equation}

The differential equations of the geodesics of the surface are given by
\begin{equation*}\label{eqn2.5.123}
\ddot{x}^1-\frac{2x^1(\dot{x}^1)^2(\dot{x}^2)^2}{2(\dot{x}^1)^2+(x^1)^2(\dot{x}^2)^2}=0,
\end{equation*}
and
\begin{equation*}\label{eqn2.5.124}
\ddot{x}^2+\frac{4(\dot{x}^1)^3(\dot{x}^2)}{x^1[2(\dot{x}^1)^2+(x^1)^2(\dot{x}^2)^2]}=0.
\end{equation*}
It is evident that $x^1= c_1t+c_2,~ x^2=k_1$ and $x^1= k_2,~ x^2=c_3t+c_4$ are solutions of the  above equations.
\end{proof}
\begin{rem}
\textnormal{First three authors of the present paper classified the minimal surfaces of revolution  in Kropina space obtained by deformation of Euclidean metric with a $1$-form \cite{RGAKBT}. However we realise that there is a mistake in the differential Equation $(25)$ of that paper. Fortunately we find the same differential equation (corrected version of Equation $(25)$ in \cite{RGAKBT}) in  \eqref{eqn2.5.1111} of the present paper and the solution of the equation is obtained too.}
\end{rem}
\subsection*{Acknowledgement}
 The first author is supported by UGC Senior Research Fellowship, India with reference no. 1076/(CSIR-UGC NET JUNE 2017) and second author is supported by CSIR Research Associate Fellowship with file number 09/0013(011312)/2021-EMR-I.

 \end{document}